 \documentclass[12pt]{article}

\usepackage{amsfonts,amsmath,amssymb,amsthm}
\usepackage{epsfig}
\usepackage{graphics}
\usepackage{graphicx}
\usepackage{times}

\title{On the Zeros of Epstein Zeta Functions \footnote{This work was supported by the Korea Research Foundation Grant funded by the Korean Government(MOEHRD, Basic Research Promotion Fund)(KRF-2008-313-C00009).     }}
\begin{document}
\newcommand\tabcaption{\def\@captype{table}\caption}
\newtheorem{thrm}{Theorem}[section]
\newtheorem{defn}[thrm]{Definition}
\newtheorem{co}[thrm]{Corollary}
\newtheorem{lemma}[thrm]{Lemma}
\newtheorem{prop}[thrm]{Proposition}
\newtheorem*{claim}{Claim}
\newtheorem*{con}{Conjecture}
\newtheorem*{ack}{Acknowledgment}
\newtheorem*{key}{Key Words}
\newtheorem*{msc}{Mathematics Subject Classification (2010)}

\numberwithin{equation}{section}

\def\bth{{\boldsymbol \theta}}
\def\bla{{\boldsymbol \lambda}}

\def\blog{{\mathbf {Log} ~} }

\def\x{{\mathbf x}}
\def\y{{\mathbf y}}

\def\B{{\mathbf B}}
\def\F{{\mathbf F}}
\def\L{{\mathbf L}}
\def\M{{\mathbf M}}

\def\m{{\mathbf m}}
\def\r{{\mathbf r}}
\def\z{{\mathbf z}}
\def\w{{\mathbf w}}

\def\n{{\mathfrak n}}
\def\p{{\mathfrak p}}
\def\N{{\mathfrak N}}

\def\R{\text{\rm{Re}}\,}
\def\I{\text{\rm{Im}}\,}
\def\s{\text{\rm sign}}
\def\bysame{\leavevmode\hbox to3em{\hrulefill}\thinspace}

\author{Yoonbok Lee\\
\\ Department of Mathematics, University of Rochester \\ {\it lee@math.rochester.edu }}

\parskip=12pt

\maketitle 

\begin{abstract}
We investigate the zeros of Epstein zeta functions associated with
a positive definite quadratic form with rational coefficients.
Davenport and Heilbronn,  and also Voronin,  proved the
existence of zeros of Epstein zeta functions off the critical
line when the class number of the quadratic form is bigger than $1$.
These authors give lower bounds for the number of zeros in strips that are
of the same order as the more easily proved upper bounds.
In this paper, we improve their results by providing asymptotic formulas for the number of zeros.
\end{abstract}

\begin{key}
  Epstein zeta functions, Hecke $L$-functions, zeros, mean motions.
\end{key}

\begin{msc}
 11M26, 11M41.
\end{msc}

\pagestyle{myheadings}

\section{Introduction}
Let $Q(x,y) = ax^2 +bxy +cy^2$ be a positive definite quadratic
form with a fundamental discriminant $d$, where $a,b,c \in
\Bbb{Z}$ and $d|D = b^2 -4ac <0$. The Epstein zeta function
associated with $Q$ is defined by
$$E(s, Q)= \sum {}^{'} Q(m,n)^{-s} \qquad(\R s>1),$$
where the sum is over all integers $m,n $ not both zero. It has
an analytic  continuation to the whole complex plane, except for the point
$s=1$,  and it satisfies the functional equation
\begin{equation*}
\left( \frac{\sqrt{-D}}{2 \pi} \right)^s \Gamma(s) E(s, Q) =\left(
\frac{\sqrt{-D}}{2 \pi} \right)^{1-s} \Gamma(1-s) E(1-s, Q).
\end{equation*}

It is well known that the non equivalent quadratic forms of
discriminant $d$ correspond one-to-one to the classes of ideals in
the quadratic field $\mathbb{Q}(\sqrt{D})$. The number of
representations of $n$ by a quadratic form is the number of
integral ideals of norm $n$ in the corresponding ideal class, times
the number $w$ of roots of unity in $\mathbb{Q}(\sqrt{D})$. Thus,
it follows that
\begin{equation*}
E(s, Q) = \frac{w}{h(D)} \sum_{\chi} \bar{\chi}(\mathfrak{a}_Q )
L(s, \chi).
\end{equation*}
Here $h(D)$ is the class number, $\mathfrak{a}_Q$ is any integral
ideal in the ideal class corresponding to the equivalence class of
$Q$, $\sum_\chi$ is a sum over all characters of the class group, and
$L(s, \chi)$ is the Hecke $L$-function defined by
$$L(s, \chi) =
\sum_{\n} \frac{\chi(\n)}{ \N (\n)^s} =
\prod_{\p} \left( 1- \frac{ \chi(
\p)}{\N( \p)^s}\right)^{-1}\qquad(\R s>1),
$$
where $\N$ is the norm. Let $P_0$ be the set of primes $p$ which
are squares of prime divisors $\p$ of
$\mathbb{Q}(\sqrt{D})$, that is,  $p= \p^2$. We note that
$\chi(\p)= \pm 1$ and $p | D$ in this case. Let $P_1$ be
the set of primes $p$ which remain prime in the
ring of integers of $\mathbb{Q}(\sqrt{D})$, that is,   $p=\p$. Also, let $P_2$ be the set
of primes $p$ which split completely in $\mathbb{Q}(\sqrt{D})$ so that
$p=\p \p'$, say. Then, we have
\begin{equation*}
L(s, \chi) = \prod_{p \in P_0 } \left( 1- \frac{
\chi(\p)}{p^s}\right)^{-1} \prod_{p \in P_1} \left( 1-
\frac{1}{p^{2s}}\right)^{-1} \prod_{p\in P_2} \left( 1- \frac{2 \R
\chi (\p)}{p^s}+\frac{1}{p^{2s}} \right)^{-1}.
\end{equation*}
Since $L(s, \chi) = L(s, \bar{\chi})$, we can write
\begin{equation*}
E(s, Q) = \sum_{j=1} ^J  a_j L(s, \chi_j)
\end{equation*}
in which $\chi_j \neq \chi_k$ and $ \chi_j \neq \bar{\chi}_k$
for $j\neq k$,  and the coefficients $a_j $ satisfy $a_j = w h(D)^{-1}\chi_j (\mathfrak{a}_Q )$
for real characters $\chi_j$, and $a_j = 2w h(D)^{-1}  \R \chi_j
(\mathfrak{a}_Q )$ for complex characters $\chi_j$.
Here $J$ is the number of real characters plus one-half   the number of complex characters
so that, in particular, $J=1$ if and only if $h(D)=1$.

When $h(D)=1$, $E(s, Q)$  has an Euler product  and  is expected to
satisfy the analogue of the Riemann hypothesis. In contrast, when
$h(D)>1$, we expect $E(s, Q)$  to behave quite
differently.
For example, Bombieri and Hejhal \cite{BH} studied the zero distribution of
linear combinations of the form
$F(s) = \sum_{j=1}^J b_j e^{i\alpha_j} L_j(s)$,
where the $L$-functions all satisfy the same functional equation, the
$\alpha_j$'s are certain real numbers related to the functional equation, and the
 $b_j$'s are arbitrary real coefficients. The Epstein zeta functions
 we are considering are of this form. Assuming the generalized Riemann
Hypothesis for $L_j (s)$ and a weak hypothesis about the spacing of
the zeros of $L_j(s)$, they proved that almost all the zeros of $F(s)$ are simple
and on the critical line $\R s=1/2$.   In addition, Hejhal \cite{H2} announced that
except for a set of $b_j$'s with small measure,
the number of zeros of $F(s)$ in $\R s \geqslant \sigma$, $T
\leqslant \I s \leqslant T+H$ is of order
$$\frac{ H}{(\sigma- \frac{1}{2}) \sqrt{\log \log T}}
$$
for
$$ \frac12 +  \frac{( \log \log T)^\kappa }{(\log T) } \leqslant
\sigma \leqslant \frac12 + \frac{1}{(\log T)^{\delta}},
$$
and $c_1 T^w
\leqslant H \leqslant c_2 T$, where $\kappa >2$, $0<\delta<1$, and $1/2<w\leqslant 1$.
This generalized his earlier result \cite{H1} for the case $J=2$ (see also Selberg~\cite{Se}).
A little further to the right of the line $\sigma=1/2$ we have the following result of
Voronin~\cite{V} (or see Chapter 7 of \cite{KV}).
\begin{thrm}[Voronin]\label{vo:1}
Let $D$ be a negative integer. Suppose that the class number of the field
$\mathbb{Q}(\sqrt{D})$ is greater than $1$ and that $Q$ is a quadratic form with
integer coefficients whose discriminant is equal to the
discriminant of $\mathbb{Q}(\sqrt{D})$. Then for any $\sigma_1 $
and $\sigma_2 $ with $1/2< \sigma _1 < \sigma_2 <1$ and
for $T$ sufficiently large, the region $\sigma_1 < \R s < \sigma_2
$, $|\I s |<T$ contains at least $cT$ zeros of $E(s, Q)$, where
$c=c(\sigma_1 , \sigma_2 , Q)>0$ does not depend on $T$.
\end{thrm}

Earlier, Davenport and Heilbronn \cite{DH} had  shown that when $h(D)>1$, $E(s, Q)$
 has infinitely many zeros in $\R s>1$.
Recently, Bombieri and Mueller \cite{BM} investigated the  zeros of some specific Epstein zeta functions.
Define $ \sigma(Q) = \sup \{ \R \rho : E(\rho , Q)=0\} $ for a quadratic form $Q$ and let $Q_1 (m,n) = m^2 + 5n^2 $ and $Q_2 (m,n) = 2m^2 + 2mn + 3n^2$. Bombieri and Mueller evaluate $\sigma(Q_1 )$ and $\sigma(Q_2)$ numerically and investigate the zeros of $E(s, Q_1)$ near the line $\R s = \sigma (Q_1)$. They also prove the following theorem.

\begin{thrm}[Bombieri and Mueller]\label{bmt:1}
Let $1 < \sigma_1 < \sigma_2 < \sigma(Q_i)$. Then the number of
zeros of $E(s, Q_i)$ in $ \sigma_1 < \R s < \sigma_2 $, $ 0< \I s
<T$ has exact order $T$.
\end{thrm}


Our main theorem improves the results of Davenport and Heilbronn, Voronin,
and  Bombieri and Mueller by providing an asymptotic formula for the number of zeros in strips.

\begin{thrm}\label{thm:1}
Assume the same hypothesis as in Theorem \ref{vo:1}. Then  $E(s, Q)$
has \,$c\,T+o(T)$ zeros in any region $\sigma_1 < \R s <\sigma_2 $, $|\I s |<T$,
with   $1/2<\sigma_1<\sigma_2$, where $c\geq 0$ and $c$
is a function of $\sigma_1 ,\sigma_2$, and $Q$.
If $\sigma_1 \leq 1$, then $c>0$.
In  the special cases   $Q_1 (m,n) = m^2 + 5 n^2$ and
$ Q_2 (m,n) = 2m^2 +2mn+3n^2$, we have $c >0$ provided that
$ 1/2< \sigma_1 <  \sigma (Q_i) \ (i=1,2)$.
\end{thrm}

Based on his work in \cite{K1}, Ki predicted Theorem \ref{thm:1} for
$Q_1$ and $Q_2$ in his AMS review of Bombieri and Mueller's paper.

We prove Theorem \ref{thm:1} in $\S 3.1$ and $\S3.2$. When $h(D) = 2$ or $3$,
the Epstein zeta function is a linear combination of two
Hecke $L$-functions. In these cases the proof is a
straightforward application of  methods  found in Borchsenius and Jessen \cite{BJ}
and we present it in $\S 3.1$. When the Epstein zeta function is a linear
combination of more than two Hecke L-functions, we are not able to
prove the positivity of the constant $c$ inside the critical strip
by our   method. Fortunately, Voronin's result (Theorem
\ref{vo:1}) guarantees  positivity in this case. We  explain this in $\S
3.2$.

We expect, but are not able to prove, that the constant $c$ in
Theorem~\ref{thm:1} is positive when
$1< \sigma_1 <\sigma_2 < \sigma(Q)$ for all positive definite binary
quadratic forms $Q$ with integer coefficients, and not just  for $Q_1$ and $Q_2$.
 However, we provide several partial results in this direction in $\S 4$.

A consequence of the proof of our main theorem is the following result.
\begin{thrm}\label{co:1}
Assume the same hypothesis as in Theorem \ref{vo:1}. Let $\sigma_0
> \frac{1}{2}$ be fixed. Then, the number of zeros $\rho$ of
$E(s,Q)$ with $\R \rho = \sigma_0 $ and $0< \I \rho <T $ is $o(T)$
as $T \rightarrow \infty$.
\end{thrm}
We give the proof at the end of $\S3.2$.

\section{Preliminaries}

We begin with a few background definitions and facts, many of which may be found in Borchsenius and Jessen~\cite{BJ}.

    Let $f(s) = f(\sigma+it)$ be almost periodic in the strip $[\alpha, \beta]$ and not identically zero, and define
  the  Jensen function of $f(s)$ by 
    $$ 
    \varphi_f (\sigma) = \lim_{T_2 - T_1 \rightarrow \infty} 
    \frac{1}{T_2-T_1} \int_{T_1} ^{T_2} \log |f(\sigma+it)|dt.
    $$
Then the convergence is uniform  in the interval $[\alpha, \beta]$, and $\varphi_f (\sigma) $ is a convex function of $\sigma$.    
    
    A distribution function  in $\mathbb{C}^J$ is a completely additive, nonnegative set function $\mu(B)$ defined for all Borel sets $B \subseteq \mathbb{C}^J$ with $ \mu(\mathbb{C}^J) < \infty$. A set $B$ is called a continuity set of $\mu$ if $\mu( \textrm{closure of } B) = \mu ( \textrm{interior of } B )$. A sequence of distribution functions 
    $\mu_n$ is said to  converge to $\mu$, written $\mu_n \rightarrow \mu$, if there exists a distribution function $\mu$ such that $\mu_n (B) \rightarrow \mu(B)$ for all continuity sets of   $\mu$. We have $\mu_n \rightarrow \mu$ if and only if $\int_{\mathbb{C}^J} h(\x) d\mu_n (\x) \rightarrow \int_{\mathbb{C}^J} h(\x) d\mu(\x)$ holds for all bounded continuous functions $h(\x)$ in $\mathbb{C}^J$.

    A distribution function $\mu_\sigma$ depending on a parameter $\sigma \in ( \alpha, \beta)$ is said to depend continuously on $\sigma$, if $\mu_{\sigma_n} \rightarrow \mu_{\sigma_0}$  
     for any sequence $\sigma_n$ in $( \alpha, \beta)$ converging to $\sigma_0$. A distribution function $\mu$ is called absolutely continuous if $\mu(B)=0$ for every Borel set $B$ of measure $0$. This is the case if and only if there exists an integrable function $F(\x)$ in $\mathbb{C}^J$ such that $\mu(B) = \int_B F(\x) d\x$ for any Borel set $B$; here $d\x$ is the Lebesgue measure $\prod_{j=1}^J dx_j dx'_j$ for $\x = ( x_1 + ix'_1 , \dots , x_J + i x'_J ) \in \mathbb{C}^J$. We call $F(\x)$ the density of $\mu$. 

    We define the Fourier transform of a distribution function $\mu$ as $ \hat{\mu}(\y) = \int_{\mathbb{C}^J} e^{i\x \cdot \y} d\mu(\x)$, where $ \x \cdot \y = \sum_{j=1}^J \{ x_j y_j + x'_j y'_j \} $ is an inner product of $\x = ( x_1+ix'_1 , \dots , x_J+ix'_J )\in \mathbb{C}^J $ and $\y =(y_1+iy'_1 , \dots, y_J+iy'_J ) \in \mathbb{C}^J $. 
 It is uniformly continuous and bounded, and the maximum of its absolute value is $\hat{\mu}(0) = \mu (\mathbb{C}^J)$. If $\hat{\mu}=\hat{\nu}$, then $\mu=\nu$. If $\mu_n \rightarrow \mu$, then $\hat{ \mu}_n (\y) \rightarrow \hat{\mu}(\y)$ holds uniformly in $ \| \y \| = \sqrt{ \y \cdot \y } \leqslant a$ for any $a > 0$; conversely, if a sequence of Fourier transforms $\hat{\mu}_n(\y)$ is uniformly convergent in $\| \y \| \leqslant a$ for any $a >0$, then the limit function is also the Fourier transform of a distribution function $\mu$, and $\mu_n \rightarrow \mu$. If the integral $\int_{\mathbb{C}^J} \| \y \|^q |\hat{\mu}(\y)|d\y$ is finite for an integer $q \geqslant 0$, then $\mu$ is absolutely continuous and its density $F(\x)$, determined by the inversion formula $F(\x) = (2 \pi )^{-J} \int_{\mathbb{C}^J} e^{-i\x \cdot \y} \hat{\mu}(\y)d\y$, is continuous and possesses continuous partial derivatives of order $\leqslant q$. Note that $\int_{\mathbb{C}^J} \| \y \|^q |\hat{\mu}(\y)|d\y$ is finite if for some $\epsilon>0$, $\hat{\mu}(\y) = O(\| \y \|^{-(2J + q + \epsilon)})$ as $\| \y \| \rightarrow \infty$.

    Now we establish a connection between Jensen functions $\varphi_{f-x}(\sigma)$ and distribution functions $\nu_\sigma$. For any $ \sigma$ and any interval $ I = [ T_1 , T_2 ]$, we define  the distribution function of 
    $f(\sigma+it)$ with respect to $ |f'(\sigma+it)|^2$ over the interval $ t \in I$  by
    $$\nu_{\sigma, I}(B) = \frac{1}{T_2 - T_1} \int_{A_{\sigma,I}(B)} |f'(\sigma+it)|^2 dt, $$
    where $A_{\sigma,I}(B)$ denotes the set of points in $t\in I$ for which $f(\sigma+it)\in B$. Then $\nu_{\sigma,I}$ converges to a distribution function $\nu_\sigma$ as $T_2 - T_1 \rightarrow \infty$. 
    We call $\nu_\sigma$ the \emph{asymptotic distribution function} of $f(\sigma+it)$ with respect to $|f'(\sigma +it)|^2 $. The following proposition,  summarizing  $\S 9 $ of \cite{BJ},  describes a relation between $\varphi_{f-x}(\sigma)$ and $\nu_\sigma$.

    \begin{prop}\label{pro:1}
    Let $f(s)$ be almost periodic in the strip $[\alpha, \beta]$ and not identically zero. Let $\nu_\sigma$ be the asymptotic distribution function of $f(\sigma+it)$ with respect to $|f'(\sigma +it)|^2 $. Suppose $\nu_\sigma$ is absolutely continuous for every $\sigma$ and its density $G_\sigma (x)$ is a continuous function of $x$ and $\sigma$. Then the Jensen function $\varphi_{f-x}(\sigma)$ is twice differentiable with   second derivative
    \begin{equation}\label{eqn:pro1}
    \varphi_{f-x}''(\sigma) = 2 \pi G_\sigma (x).
    \end{equation}
    \end{prop}

Taking   $f(s)$ equal to  $E(s,Q)$ or $L(s,\chi)$, we see that Proposition~\ref{pro:1} only applies 
when $\sigma>1$ because $E(s,Q)$ and $L(s,\chi)$ are almost periodic only  in this half-plane. The main difficulty we face in our proof of Theorem~\ref{thm:1} is to show that \eqref{eqn:pro1} 
also holds in the half-plane $ \sigma > 1/2$.

 To state our next proposition we require the following definition.  
    \begin{defn}
    Suppose that $r>0$ and that $f(s), f_1 (s), f_2 (s) , \dots$ are functions defined in the half strip $\alpha<\R s <\beta$, $\I s > 1$, where $\beta$ may be $\infty$. Then we say that $f_n (s)$ converges in the mean with index $r$ towards $f(s)$ in $[\alpha, \beta]$ if and only if 
    $$ \limsup_{T \to \infty} \frac{1}{T} \int_1 ^T \int_{\alpha_1} ^{\beta_1} | f(\sigma+it)-f_n (\sigma+it)|^r d\sigma dt \rightarrow 0$$
    as $n \to \infty$ for any interval $\alpha<\alpha_1 < \sigma<\beta_1<\beta$.
    \end{defn}

    \begin{prop}\label{pro:2}
    Let $L (s, \chi )$ be a Hecke L-function  with $ \chi$  any character of the ideal class group of $\mathbb{Q} (\sqrt{D})$, and let  $E(s, Q)$ be the Epstein zeta function associated with  a positive definite binary quadratic form $Q$ with integer coefficients.
Let $L_n (s, \chi) = \prod_{ \N(\p)\leqslant p_n} (1- \chi(\p) \N (\p)^{-s})^{-1}$ and 
    $E_n (s, Q) = \sum_{j=1}^J a_j L_n (s , \chi_j)$, where $p_n$ is the n-th prime number. 
    Then $L_n (s, \chi)$ and $E_n (s, Q)$ converge in mean with  index 2 
    towards $L(s, \chi)$ and $E(s, Q)$, respectively, in $[ 1/2 , \infty]$.
   \end{prop}

\begin{proof}
We prove that $L_n (s,\chi)$ converges in mean with index 2 to $L(s,\chi)$ in $[1/2, \infty]$. Then the convergence of $E_n ( s, Q)$ follows immediately. Based on the approximate functional equation for $L(s,\chi)$, we have
$$
\int_1^T \left| L(\sigma+it ,\chi) - \sum_{m \leqslant X} \frac{b_m }{ m^{\sigma+it}}  \right|^2 dt \ll T^{2-2\sigma + \epsilon} 
$$
for any $\epsilon>0$ and  $\sigma > 1/2$, where $ X = c t $, $ c=  \sqrt{|D|} /( 2 \pi )$, and $L(s,\chi) = \sum_{m=1}^\infty b_m m^{-s}$. Note that $b_m = \sum_{ \N (\n ) = m } \chi( \n ) $ and $ | b_m | \leqslant d(m) \ll m^\epsilon$ for any $\epsilon>0$. We also have 
$$
\int_1^T \left| \sum_{m \leqslant p_n} \frac{b_m }{ m^{\sigma+it} } - \sum_{m \leqslant X} \frac{b_m }{ m^{\sigma+it}} \right|^2 dt \ll T p_n ^{1-2\sigma+\epsilon} + T^{2-2\sigma+\epsilon}.
$$
See p.280--282 of \cite{KV} for the details. Thus we have
\begin{equation}\label{pro1:eq1}
\limsup_{T \to \infty} \frac1T  \int_1^T \left| L(\sigma+it ,\chi) - \sum_{m \leqslant p_n} \frac{b_m }{ m^{\sigma+it}}  \right|^2 dt  \leqslant C  p_n ^{1-2\sigma+\epsilon} 
\end{equation}
for some constant $C>0$ and any $ \sigma > 1/2$. 

Observe that
$$
 L_n (s,\chi) - \sum_{\N( \m ) \leqslant p_n} \frac{ \chi( \m) }{ \N(\m) ^s} = \sum_{\substack{ \N ( \m ) > p_n \\  \p | \m \Rightarrow \N ( \p ) \leqslant p_n   }}  \frac{ \chi(\m)  }{ \N ( \m )^s}  = \sum_{ \substack{ m > p_n \\ p|m \Rightarrow p \leqslant p_n   }}  \frac{b_m}{m^s}
$$
and
$$
\int_1^T \left|   \sum_{ \substack{ m > p_n \\ p|m \Rightarrow p \leqslant p_n   }}  \frac{b_m}{m^{\sigma+it }} \right|^2 dt =  \sum_{ \substack{ m > p_n \\ p|m \Rightarrow p \leqslant p_n   }}  \frac{|b_m |^2}{m^{2\sigma }}( T+O(m)) \ll T p_n ^{1-2 \sigma + \epsilon} + \prod_{p \leqslant p_n } ( 1- p^{1-2\sigma+\epsilon} )^{-1} 
$$
for any $\epsilon >0$. Thus we have
\begin{equation}\label{pro1:eq2}
\limsup_{T \to \infty} \frac1T  \int_1^T \left| L_n (\sigma+it ,\chi) - \sum_{m \leqslant p_n} \frac{b_m }{ m^{\sigma+it}}  \right|^2 dt  \leqslant C  p_n ^{1-2\sigma+\epsilon} .
\end{equation}
We complete the proof by combining \eqref{pro1:eq1} and \eqref{pro1:eq2}.

\end{proof}

    To prove  Theorem~\ref{thm:1} we require the following sequence of lemmas.  
    
        \begin{lemma}\label{fin:1}
    Let $ \sigma >1/2$ and let $M$ and $K$ be fixed positive integers. 
    Let $\z = (z_1 , \dots, z_K ) \in \mathbb{C}^K $ and 
    $ \w = ( w_1 , \dots, w_K)\in \mathbb{C}^K$. Define
    $$  \mathcal{I} ( \z, \w) = \int _0^1 \cdots \int_0 ^1  \prod_{k=1}^K \prod_{m=1}^n \left( 1- \frac{ c(m,k) e^{2 \pi i \theta_m }}{p_m ^{\sigma+z_k}}\right)^{-1}\left( 1- \frac{ \overline{c(m,k)} e^{-2 \pi i \theta_m }}{p_m ^{\sigma+w_k}}\right)^{-1} d\theta_1 \cdots d\theta_n , $$
    where $|c(m,k)|\leqslant 1$. Then  there is a constant $A$, depending only on $M, K,$ and $\sigma$, 
    such that
    $$ \left| \prod_{k=1}^K  \frac{ \partial^{m_k + n_k} \mathcal{I}  }{\partial z_k ^{m_k} \partial w_k ^{n_k}} (0,0) \right| \leqslant A,$$
   whenever $ \sum_{k=1}^K ( m_k + n_k ) \leqslant M $.
    \end{lemma}

    \begin{proof}
    \begin{equation*}\begin{split}
    \mathcal{I} ( \z, \w) = &  \prod_{m=1}^n \int_0 ^1  \prod_{k=1}^K  \left( 1- \frac{ c(m,k) e^{2 \pi i \theta }}{p_m ^{\sigma+z_k}}\right)^{-1}\left( 1- \frac{ \overline{c(m,k)} e^{-2 \pi i \theta }}{p_m ^{\sigma+w_k}}\right)^{-1} d\theta \\
    = & \prod_{m=1}^n \int_0 ^1 \Bigg(\sum_{l_1 , \cdots , l_K = 0}^{\infty} \frac{ c(m,1)^{l_1} \cdots c(m,K)^{l_K}e^{ 2 \pi i \theta ( l_1 + \cdots + l_K )}  } { p_m ^{\sigma ( l_1 + \cdots + l_K ) + ( z_1 l_1 + \cdots +  z_K l_K ) }} \Bigg)  \\
     & \qquad \qquad \times \Bigg(\sum_{r_1 , \cdots , r_K = 0}^{\infty}  \frac{  \overline{c(m,1)}^{r_1} \cdots \overline{c(m,K)}^{r_K}e^{- 2 \pi i \theta ( r_1 + \cdots + r_K )}  } { p_m ^{\sigma ( r_1 + \cdots + r_K ) + ( w_1 r_1 + \cdots +  w_K r_K ) }}\Bigg)  d\theta \\
     = & \prod_{m=1}^n \left( 1+ \sum_{k=1}^\infty \sum_{\substack{ l_1 + \cdots + l_K = k \\  r_1 + \cdots + r_K = k } }      \frac{c(m,1)^{l_1} \cdots c(m,K)^{l_K}  \overline{c(m,1)}^{r_1} \cdots \overline{c(m,K)}^{r_K} }{p_m ^{2\sigma k +   ( z_1 l_1 + \cdots +  z_K l_K ) + ( w_1 r_1 + \cdots +  w_K r_K )                     }}\right).
    \end{split}\end{equation*}
    Let $0 < \epsilon < \sigma- 1/2$ and $ \| \z \|, \| \w \|  \leqslant \epsilon$. Since there are $ {\binom {K+k-1}{k}}^2$ solutions to the system of equations
    $l_1 + \cdots + l_K =    r_1 + \cdots + r_K = k $, we have
    \begin{equation*}
     | \mathcal{I} (\z, \w) | \leqslant \prod_{m=1}^\infty \left( 1+ \sum_{k=1}^\infty \frac{ 1 }{p_m ^{2k(\sigma - \epsilon)}} {\binom {K +k-1}{k}}^2   \right) \leqslant  A_K \zeta ( 2 ( \sigma - \epsilon))^{K^2} .
    \end{equation*}
 Applying Cauchy's integral formula to each variable $z_k $   and $w_k$  on 
 a circle of radius $\epsilon/K$, we obtain the result.
    \end{proof}

    \begin{lemma}\label{curve:1}
    Let $\delta$ be fixed positive number and $\y = (y_1+iy'_1, \dots , y_J +iy'_J ) \in \mathbb{C}^J$. 
    For each $j\leq J$, let   $f_j(x) = a_j x+a_{j,2}x^2 +a_{j,3}x^3+ \cdots $ be 
    a holomorphic function in $|x| < \rho_j$  whose first coefficient  $a_j$ is real.
    Define
    $$g (\y, \theta) = \sum_{j=1}^J ( y_j \R f_j (re^{2\pi i \theta}) + y'_j \I f_j (re^{2\pi i \theta})) $$
    for $\theta \in [0,1]$. Then  for $0<r< \min_j \rho_j $, we have
    $$ \left| \int_0 ^1 e^{ig (\y, \theta)} d\theta \right| \leqslant \frac{C}{\sqrt{r \| \y \| }}$$
    in the region  $| \sum_{j=1}^J a_j (y_j+i y'_j) | \geqslant \delta \| \y \|$ as $ \| \y \| \to \infty$.
    \end{lemma}

    \begin{proof}
    We have
    \begin{equation*}\begin{split}
    g (\y ,  \theta ) & =  r \left( \sum_{j=1}^J a_j  y_j \cos ( 2 \pi \theta) +  a_j y'_j \sin ( 2 \pi \theta ) \right)+O(r^2 \| \y \|) \\ & = r \left| \sum_{j=1}^J a_j (y_j+iy'_j ) \right| \cos ( 2 \pi (\theta - \xi ) )+O(r^2 \| \y \| )
    \end{split}\end{equation*}
    for some $\xi$. Now, if $g ( \y, \theta )=0$, then
    $$ \theta = \xi \pm \frac14 + O\bigg( \frac{\| \y \| r}{|\sum_{j=1}^J a_j (y_j+iy'_j ) |}\bigg)
    = \xi  \pm \frac{1}{4} + O(r ), $$
    where $| \sum_{j=1}^J a_j (y_j+iy'_j ) | \geqslant \delta \| \y \|$. Similarly, if $g'( \y , \theta')=0$, then $\theta' =  \xi  + O(r )$ or $\xi  +1/2+O(r )$, and if $g''(\y, \theta'')=0$, then $\theta'' = \xi \pm \frac{1}{4} + O(r ) $. When $\theta$ is close to $\xi$, $ g''(\y, \theta) \neq 0$. Thus, $ g'(\y, \theta)$ has only two zeros modulo 1. A similar argument  argument shows that  $g''(\y, \theta)$  also has only two zeros modulo 1.

    Let $I_k = [ \xi + (2k-1)/ 8 , \xi + (2k+1)/8]$ for $k=1,2,3,4$. Then, we have
    \begin{equation*}\begin{split}
    \left| \int_{I_1 \cup I_3 } e^{i g(\y, \theta)} d\theta \right| & \ll  \frac{1}{ \min \{ |g'(\y, \theta)| : \theta \in I_1 \cup I_3 \} } \ll \frac{1}{r|\sum_{j=1} ^J a_j ( y_j+iy'_j) |}, \\
    \left| \int_{I_2 \cup I_4 } e^{i g( \y ,\theta)} d\theta \right| & \ll  \frac{1}{ \min \{ \sqrt{ |g''(\y, \theta)|} : \theta \in I_2 \cup I_4 \} } \ll \frac{1}{\sqrt{r|\sum_{j=1} ^J a_j ( y_j +iy'_j ) |}},
    \end{split}\end{equation*}
    by Lemmas 4.2 and 4.4 of \cite{Ti}. Adding these estimates, we obtain  the result.
    \end{proof}

    The next lemma shows that prime ideals are equally distributed in each ideal class.
    \begin{lemma}\label{tauber:1}
    For any ideal class $\mathcal{C}$ of $\mathbb{Q}(\sqrt{D})$, we have
    $$    \sum_{ \substack{ \p \in \mathcal{C} \\ \N(\p) \leqslant x \\ \N( \p)  \in P_2 }} 1 = \frac{1}{h(D)} \frac{x}{\log x} + o \left( \frac{x}{\log x} \right) $$
    as $x \rightarrow \infty$; here $P_2$ is the set of rational primes that split completely in 
    $\mathbb{Q}(\sqrt{D})$.
    \end{lemma}
    This lemma can be proved by Tauberian theorems. Since it is quite standard, we omit the proof. The lemma following Definition \ref{def:1} is known as Kronecker's theorem. See $\S A.8 $ \cite{KV} for the proof.

    \begin{defn}\label{def:1}
    Let $\gamma : [0,\infty) \to \mathbb{R}^N $ be continuous. We say that the curve $\gamma(t)$ is uniformly distributed mod 1 in $\mathbb{R}^N$ if the following relation holds for every parallelepiped $\Pi = [\alpha_1,\beta_1 ]\times \cdots \times [\alpha_N, \beta_N ]$, $0 \leqslant  \alpha_j <\beta_j \leqslant 1$ for $j = 1, \dots , N$:
    $$  \lim _{T\rightarrow \infty} \frac{1}{T} | \{ t \in [0,T]: \gamma(t) \in \Pi \mod 1 \}| = \prod_{j=1}^N (\beta_j -\alpha_j ).$$
    Here, $ \gamma(t) \in \Pi \mod 1 $ means that $ \gamma(t) - \y \in \Pi$ for some $ \y \in \mathbb{Z}^N$.
    \end{defn}

    \begin{lemma}\label{krt:1}
    Suppose that the curve $\gamma(t)  \in \mathbb{R}^N$ is uniformly distributed mod 1 in $\mathbb{R}^N$,   that $\F(\x)$ is Riemann integrable on the unit cube $[0,1]^N$, and $\F(\x)= \F(\x+\m)$ for any $ \x \in \mathbb{R}^N$ and $ \m \in \mathbb{Z}^N$. Then 
    $$ \lim_{T \to \infty} \frac{1}{T} \int_0 ^T \F(  \gamma (t)   ) dt = \int_0 ^1 \cdots \int_0 ^1 \F(x_1 , \dots, x_N ) dx_1 \cdots dx_N .$$
    In particular, we can choose $\gamma(t)= \gamma_N (t) = ( - (t \log p_1) /(2 \pi) , \dots, -(t\log p_N ) / ( 2 \pi) )$.
    \end{lemma}

\section{Proof of Theorems~\ref{thm:1} and~\ref{co:1}}

As we mentioned in $\S1$, the proof of Theorem~\ref{thm:1} splits naturally into two parts: the case when 
$E(s, Q)$  is a sum of exactly two $L$-funtions, and when it is a sum of more than two.
We handle these two cases separately in $\S3.1$ and $\S3.2$. 
 In both cases we require the following result, which  is Theorem 1 in \cite{BJ}.
 Note that the definition of the Jensen function here is slightly different from our previous one.
    
    \begin{thrm}\label{bjt:1}
        Let $-\infty \leqslant \alpha < \alpha_0 <\beta_0 <\beta \leqslant \infty$ and let $f_1 (s), f_2 (s), \dots $ be a sequence of functions almost periodic in $[\alpha, \beta]$ converging uniformly in $[\alpha_0 , \beta_0]$ to  a function $f(s)$. Suppose that none of the $f_n$'s is identically zero,  that $f(s)$ has an analytic continuation to  the half strip $\alpha < \sigma < \beta$, $t >t_0$, and, finally,  that $f_n (s)$ converges in mean 
 to  $f(s)$ in $[\alpha, \beta]$ for some index $p>0$. For  any $T_0>t_0$, define the Jensen function 
of $f(s)$ to be
  $$ 
  \varphi_f (\sigma) = \lim_{T    \rightarrow \infty} \frac{1}{T-T_0} \int_{T_0} ^{T } \log |f(\sigma+it)|dt.
  $$
 (Note that this is independent of the choice of $T_0$.)
 Then the convergence is uniform  in the interval $[\alpha, \beta]$,  
  $\varphi_{f_n}(\sigma)$ converges uniformly to $\varphi_f (\sigma)$ in $[\alpha,\beta]$  as $n\rightarrow \infty$, and   $\varphi_f (\sigma)$ is convex in $(\alpha, \beta)$.
        For every strip $(\sigma_1 , \sigma_2 )$, where $\alpha< \sigma_1 < \sigma_2 < \beta$, and for $T_0 > t_0$, the two relative frequencies of zeros defined by
        $$ \underline{H}_f (\sigma_1 , \sigma_2 ) = \liminf_{T \rightarrow \infty}\frac{ N_f (\sigma_1 ,\sigma_2 ;T_0, T)}{T-T_0},$$
        $$ \overline{H}_f (\sigma_1 , \sigma_2 ) = \limsup_{T \rightarrow \infty}\frac{ N_f (\sigma_1 ,\sigma_2 ;T_0, T)}{T-T_0},$$
        where $N_f (\sigma_1, \sigma_2; T_0, T)$ denotes the number of zeros of $f(s)$ in the rectangle $\sigma_1 < \sigma < \sigma_2 $, $T_0 < t < T$, satisfy the inequalities
        \begin{equation} \label{kre:1}
            \frac{1}{2\pi} ( \varphi'_f (\sigma_2 - )- \varphi'_f (\sigma_1 +)) \leqslant \underline{H}_f (\sigma_1 , \sigma_2 ) \leqslant \overline{H}_f (\sigma_1, \sigma_2 ) \leqslant \frac{1}{2\pi } ( \varphi'_f (\sigma_2 + )- \varphi'_f (\sigma_1 -)).
        \end{equation}
    \end{thrm}

\subsection{Sums of two Hecke $L$-functions}

When $h(D)= 2$ or $3$,  $E(s,Q)$ is a linear combination of two Hecke $L$-functions. 
We assume this to be the case and write $E(s, Q) = a_1 L(s, \chi_1 ) + a_2 L(s, \chi_2 )$, where $\chi_1$ is the principal character, $\chi_2 \neq \chi_1 $, and neither $a_i$ is $0$. By Proposition \ref{pro:2} and Theorem \ref{bjt:1},  for any fixed $ 1/2 < \sigma_1 < \sigma_2 $, the Jensen functions $ \varphi_{E_n}$ and $\varphi_{L_n } $, where $E_n (s,Q) $ and $L_n (s,\chi_j) $ are defined in Proposition \ref{pro:2}, converge uniformly in $[\sigma_1, \sigma_2]$ to  $\varphi_E$ and $\varphi_L$, respectively. By direct calculation, we see that
    $$\varphi_{L_n}(\sigma) = \R  \left[ \lim_{T \rightarrow \infty} \frac{1}{T} \int_1 ^T \log L_n (\sigma+it, \chi)dt \right] =  0,$$
      so that $\varphi_L ( \sigma ) = 0 $ for $ \sigma > 1/2$. Thus,
    \begin{equation}\label{jee:1} \begin{split}
    \varphi_E (\sigma) = & \lim_{T \rightarrow \infty} \frac{1}{T} \int_1 ^T \log \left| \frac{L(\sigma+it, \chi_2 )}{L(\sigma+it, \chi_1 )}+ a \right| dt + \log |a_2| =\varphi_{h+a} (\sigma ) + \log |a_2 |, \\
    \varphi_{E_n} (\sigma) = & \lim_{T \rightarrow \infty} \frac{1}{T} \int_1 ^T \log \left| \frac{L_n (\sigma+it, \chi_2 )}{L_n (\sigma+it, \chi_1 )}+ a \right| dt + \log |a_2|= \varphi_{h_n +a} (\sigma ) + \log |a_2 |,
    \end{split}\end{equation}
    where $h( s) = L( s, \chi_2 ) / L (s , \chi_1 )$, $h_n( s) = L_n( s, \chi_2 ) / L_n ( s , \chi_1 )$, and $ a = a_1 a_2 ^{-1}$. We want to show that $\varphi_E $ (equivalently $\varphi_{h+a }$) has a continuous derivative for $\sigma > 1/2$. For then, by (\ref{kre:1}),  we would have
    \begin{equation}\label{kre:2}
    N_E (\sigma_1 , \sigma_2 ; 0,T)= \frac{T}{2 \pi}(\varphi'_E(\sigma_2)-\varphi'_E(\sigma_1 )) + o(T).
    \end{equation}

    Let $\mu_{n,\sigma}$ be the asymptotic distribution function of $h_n(\sigma+it)$ with respect to $|h'_n(\sigma+it)|^2$. Since $h_n (s)$ is almost periodic in $ \sigma>0$,   Proposition \ref{pro:1} implies that
    \begin{equation}\label{jee:2}
    \varphi''_{h_n -x}(\sigma) = 2 \pi H_{n, \sigma}(x),
    \end{equation}
    provided that $\mu_{n,\sigma} $ is   absolutely continuous for each $\sigma>0$ and  each $n$, and that  
      its density $H_{n, \sigma}(x)$ is continuous in $x$ and $\sigma$. Define 
      \begin{equation}\label{def:L,theta}
      L_n (\sigma, \chi_j ; \bth) = \prod_{m \leqslant n} f_{m, \sigma, j}(\theta_m )  ,
      \end{equation}
       where $ \bth = ( \theta_1 , \dots, \theta_n )$ and
    \begin{equation}\label{def:f}
f_{m, \sigma, j}( \theta ) =
\begin{cases}
 ( 1- \chi_j (\p_m) p_m ^{-\sigma} e^{2 \pi i\theta})^{-1}    & if\ \    p_m = \p_m ^2 , \cr
 (1-  p_m ^{-2\sigma }e^{4 \pi i \theta})^{-1}                   & if\ \    p_m = \p_m ,     \cr
 (1-  2  (\R \chi_j ( \p_m )) p_m ^{-\sigma}  e^{2\pi i \theta } +  p_m ^{- 2 \sigma}e^{4 \pi i \theta })^{-1} & if\ \  p_m = \p_m \p'_m , \p_m \neq \p'_m .\cr
\end{cases}
\end{equation}
   Then  we have $L_n (\sigma+it, \chi_j ) = L_n (\sigma, \chi_j ; \gamma_n (t))$, where $ \gamma_n (t)$ is defined in Lemma \ref{krt:1}. By Lemma \ref{krt:1} and (\ref{jee:1}), we have
    \begin{equation*}
    \varphi_{E_n} (\sigma ) = \int_{I^n} \log \left| \frac{L_n (\sigma, \chi_2; \bth)}{L_n (\sigma, \chi_1 ; \bth)} + a \right| d\bth + \log | a_2 |,
    \end{equation*}
    where $I^n = [0,1]^n $, $\bth = (\theta_1 , \theta_2, \dots, \theta_n )$, and $  d \bth = d\theta_1 \cdots d\theta_n$. 
    
    Next we define a distribution function by
    \begin{equation}\label{def:nu}
    \nu_{n, \sigma}(B) = \int_{h^{-1}_{n, \sigma}(B)} \left| h'_{n, \sigma}(\bth) \right|^2 d\bth,
    \end{equation}
    where  $ h_{n, \sigma}(\bth)= L_n (\sigma, \chi_2 ; \bth) / L_n (\sigma, \chi_1 ; \bth)$, and 
   let  $G_{n,\sigma}$ denote  its density function. Here $h_{n, \sigma}'
   = \frac{\partial}{\partial \sigma}h_{n, \sigma}$. 
   Observe that by Lemma \ref{krt:1},
    \begin{equation*}\begin{split}
    \widehat{\mu_{n, \sigma}} (y)= & \int_{\mathbb{C}} e^{i x\cdot y} d\mu_{n, \sigma}(x)
    =  \lim_{ T \to \infty} \frac{1}{T} \int_1 ^T  e^{i h_n (\sigma+it)\cdot y }|h'_n (\sigma+it)|^2 dt \\
    = & \int_{I^n} e^{i h_{n, \sigma}(\bth)\cdot y} |h'_{n,\sigma}(\bth)|^2 d\bth
    =  \widehat{\nu_{n, \sigma}}(y),
    \end{split}\end{equation*}
     where $ x \cdot y = \R x  \R y + \I x  \I y$ for $ x,y \in \mathbb{C}$.
  It follows that
      \begin{equation}\label{jee:3}
    \mu_{n, \sigma} = \nu_{n, \sigma} \quad \mathrm{and} \quad H_{n,\sigma} = G_{n,\sigma}.
    \end{equation}     
 Similarly, for $g_{n, \sigma}(\bth) = \log L_n (\sigma, \chi_2 ; \bth) - \log L_n (\sigma, \chi_1 ; \bth)$ 
we define
    \begin{equation}\label{def:lambda}
    \lambda_{n, \sigma}(B)= \int_{g_{n, \sigma}^{-1}(B)} \left|  g'_{n, \sigma}(\bth) \right|^2 d\bth .
    \end{equation}

By applying Theorems 5--9 of \cite{BJ} to $\lambda_{n,\sigma}$ and $ \nu_{n,\sigma}$ 
in a straightforward way, we obtain the following two theorems.
    \begin{thrm}\label{bjt:2}
    For any $\sigma>0$, the distribution functions $\lambda_{n, \sigma}$ are absolutely continuous for $n \geqslant n_0$ with continuous densities $F_{n, \sigma}(x)$ which possess continuous partial derivatives of order less than $q$ for $n\geqslant n_q$.
    If $\sigma>1/2$, $\lambda_{n, \sigma}$ converges as $n \to \infty$ to some distribution function $\lambda_\sigma$ which is absolutely continuous with continuous density $F_\sigma (x)$ possessing continuous partial derivatives of arbitrarily high order. The functions $F_{n, \sigma}(x)$ and their partial derivatives converge uniformly to $F_\sigma (x)$ and its partial derivatives as $n \to \infty$. 
    If $1/2 < \sigma \leqslant 1$, then $F_\sigma(x) >0$ for all $x$. If $1/2<\sigma <1$, then $F_\sigma( x_1 + i x_2 )$ is an entire 
    function of the two variables $x_1$ and $x_2$.  The distribution functions all depend continuously on $\sigma$, and their densities and the partial derivatives of the densities are continuous functions of $x$ and $\sigma$ together. Further, if $\frac{1}{2}< \alpha <\beta$, the convergence of $F_{n,\sigma}(x)$ and their partial derivatives is uniform in $x$ and $\sigma$ together for all $x$ and $\alpha \leqslant \sigma \leqslant \beta$. If $c >0$ is arbitrary and $1/2< \alpha<\beta $, the functions $ F_\sigma (x)$ and $ F_{n, \sigma}(x) $, $n\geqslant n_0$ have a majorant of the form $K_0 e^{- c |x|^2}$ for $\alpha \leqslant \sigma \leqslant \beta$, and for every $q$ the partial derivatives of $F_\sigma(x)$ and $F_{n, \sigma}(x)$, $n \geqslant n_q$ of order less than $q$, have a majorant of the form $K_q e^{-c |x|^2 }$.
    \end{thrm}

    \begin{thrm}\label{bjt:3}
    For $\sigma>0$ the distribution functions $\nu_{n, \sigma}$ defined in \eqref{def:nu} are of the form
     $\nu_{n, \sigma}(B) = \int_{B_{\log}} e^{2 \R x} d\lambda_{n, \sigma}(x)$, where $ B_{\log} = \{ x \in \mathbb{C} | e^x \in B \} $ and  $\lambda_{n, \sigma}(x)$ is defined in \eqref{def:lambda}. For $n \geqslant n_0$ they are absolutely continuous with continuous densities $G_{n, \sigma}(x)$ that are zero if $x=0$  and are given by $G_{n, \sigma}(x) = \sum_{e^z = x} F_{n, \sigma}(z)$ if $x \neq 0$ . For $n \geqslant n_q$ the densities possess continuous partial derivatives of order less than $q$. If $\sigma>1/2$,  the distribution functions $\nu_{n, \sigma}$ converge for $n \to \infty$ to a distribution function $\nu_\sigma$   given  by $\nu_\sigma (B) = \int_{B_{\log} } e^{2 \R x} d\lambda_{\sigma}(x)$. The latter function   is absolutely continuous and has a continuous density $G_\sigma (x)$ which is zero if $x=0$  and, if $x\neq 0$, is given by $G_\sigma (x) = \sum_{e^z =x} F_\sigma (z)$. The density $G_{\sigma}(x)$ possesses continuous partial derivatives of arbitrarily high order which all vanish for $x=0$. The functions $G_{n, \sigma}$ and their partial derivatives converge uniformly to  $G_\sigma (x)$ and its partial derivatives when $n \to \infty$. If $1/2 < \sigma \leqslant 1$, then $G_\sigma(x)>0$ for all $x \neq 0$. If $1/2 < \sigma < 1$, then $G_\sigma (e^{x_1 +ix_2})$ is an entire function  of the two variables $x_1$, $x_2 $. The distribution function $\nu_{\sigma}$ depends continuously on $\sigma$  and its density $G_{\sigma}$, and the partial derivatives of $G_{\sigma}$ are continuous functions of $x$ and $\sigma$ together.  Further, if $1/2 < \alpha< \beta$, then the convergence of the $G_{n, \sigma}(x)$ and their partial derivatives to  $G_\sigma (x)$ and its partial derivatives is uniform in $x$ and $\sigma$ together for all $x$ and $ \alpha \leqslant \sigma \leqslant \beta$. If $c >0$ is arbitrary  and $1/2 < \alpha<\beta$, then the functions $G_\sigma (x)$ and $G_{n, \sigma}(x)$, $n \geqslant n_0$, have for $\alpha \leqslant \sigma \leqslant \beta$ and for $ x \neq 0$ a majorant of the form $K_0 e^{- c \log^2 |x|}$. Moreover,   for every $q$ and  $n \geqslant n_q$,  the partial derivatives of $G_\sigma (x)$ and $G_{n, \sigma}(x)$  of order less than $q$   have for $x\neq 0$ a majorant of the form $K_q e^{-c \log^2 |x|}$.
    \end{thrm}

    By (\ref{jee:2}), (\ref{jee:3}) and Theorem \ref{bjt:3}, $\varphi''_{h_n +a} (\sigma)= 2 \pi G_{n,\sigma} (-a)$ and it converges to $ 2 \pi G_\sigma (-a)$ uniformly for $ \sigma \in [\sigma_1 , \sigma_2 ]$. Thus, we have $\varphi''_{E}(\sigma) = \varphi''_{h + a } (\sigma)= 2 \pi G_\sigma (-a)$ and, by \eqref{kre:2},
    \begin{equation}\label{eqn:4}
    N_E (\sigma_1 , \sigma_2 ; 0, T) =  T \int_{\sigma_1 }^{\sigma_2 }G_\sigma (-a) d\sigma +o(T).
    \end{equation}
    This gives the first assertion of  Theorem~\ref{thm:1} when the Epstein zeta function is a linear combination of two Hecke $L$-functions. The second assertion, that the integral here is positive 
    (again for a sum of two $L$-functions),  follows because by Theorem \ref{bjt:3} $G_\sigma (-a) >0$ for $ \frac{1}{2}< \sigma \leqslant 1$. 
  The proof of the final assertion of  Theorem~\ref{thm:1}
 concerning the special positive definite binary quadratic forms    $ Q_1$ and $ Q_2$ follows from 
 Theorem \ref{bmt:1} and the fact that (\ref{eqn:4}) holds for $ 1/2 <     \sigma_1 < \sigma_2 $.   
 This concludes the proof of Theorem~\ref{thm:1} when the Epstein zeta function is a linear combination of two Hecke $L$-functions. 
 

\subsection{Sums of at least three Hecke $L$-functions}

    Now we consider the case of Epstein zeta functions of the form $E(s,Q) = \sum_{j=1}^J a_j L (s, \chi_j)$ with $ J \geqslant 3$. By Proposition \ref{pro:2} and Theorem \ref{bjt:1}, $\varphi_{E_n -x} (\sigma)$ converges to $\varphi_{E-x}(\sigma)$ uniformly for $ \sigma_1 \leqslant \sigma \leqslant \sigma_2  $ as $ n \to \infty$.

    Define $\mu_{n, \sigma}$ to be the asymptotic distribution function of $E_n (\sigma+it)$ with respect to $|E_n '
    (\sigma+it)|^2 $. We also define 
    \begin{equation}\label{eqn:ep1}
    E_{n, \sigma} (\bth) =  \sum_{j=1}^J a_j L_n (\sigma, \chi_j ; \bth)
    \end{equation}
     on $\bth \in  [0,1]^n$, and
     \begin{equation}\label{eqn:nudef} 
    \nu_{n,\sigma}(B)= \int_{E_{n,\sigma}^{-1}(B)}|E'_{n,\sigma}(\bth)|^2 d\bth
    \end{equation}
     for  Borel sets $B \subseteq \mathbb{C}$, where $E'_{n,\sigma} = \frac{\partial}{\partial \sigma} E_{n, \sigma}$. By the same argument that leads to (\ref{jee:3}), we have $\mu_{n,\sigma} = \nu_{n,\sigma}$.

    To prove Theorem ~\ref{thm:1}  for sums of three or more Hecke $L$-functions, we begin with the following observation. 
    Suppose that the distribution function $\nu_{n,\sigma}$ is absolutely continuous for each $\sigma$ and that its density $G_{n,\sigma} (x)$ is
    continuous in both $x$ and $\sigma$. By Proposition \ref{pro:1}, it would then follow that $\varphi''_{E_n -x}(\sigma) = 2 \pi
    G_{n,\sigma}(x)$. If, in addition, we knew  that for each $x$, $G_{n,\sigma}(x)$ converges to $G_\sigma (x)$ uniformly for $\sigma_1 \leqslant \sigma \leqslant \sigma_2$
    as $n \to \infty$, then we would also have that $\varphi''_{E-x}(\sigma) = 2 \pi G_\sigma (x)$. 
    Thus, by Theorem~\ref{bjt:1},
    $$N_{E-x} ( \sigma_1, \sigma_2 ; 0,T) =  T \int_{\sigma_1}^{\sigma_2} G_\sigma (x) d\sigma +o(T).$$
    Putting $x=0$, we would then have from Voronin's result (Theorem \ref{vo:1}) that
    \begin{equation}\label{eqn:5}
    \lim_{T \rightarrow \infty} \frac{1}{T} N_E ( \sigma_1 ,  \sigma_2 ; 0,T) =  \int_{\sigma_1}^{\sigma_2} G_\sigma (0) d\sigma >0,
    \end{equation}
    where $ \frac{1}{2} < \sigma_1 < \sigma_2 < 1$.
Hence, to complete our proof,  it suffices to confirm the three assumptions above. We will prove that 
$\widehat{\nu_{n, \sigma}} (y)$ 
 converges uniformly on $|y|\leqslant a$ and $\sigma_1 \leqslant \sigma \leqslant \sigma_2$ for any $a>0$, and that there are $K, d, n_0 , \epsilon>0$ satisfying 
 \begin{equation}\label{eqn:a2}
 |\widehat{\nu_{n,\sigma}}(y)| \leqslant K |y|^{-(2+\epsilon)}
\end{equation} 
 for $|y| \geqslant d $, $n \geqslant n_0$ and $ \sigma_1 \leqslant \sigma \leqslant \sigma_2 $. The equation \eqref{eqn:a2} and the facts about the Fourier transform of a distribution function in the beginning of $\S 2$ imply the first assumption on $\nu_{n,\sigma}$. 
By the inversion formula $ G_{n,\sigma} (x) = (2 \pi)^{-1} \int_{\Bbb{C}} e^{-i x\cdot y } \widehat{\nu_{n,\sigma}} (y) dy$, the equation \eqref{eqn:a2} and the uniform convergence of $\widehat{\nu_{n,\sigma}}(y)$ for any $|y| \leqslant a $ and $\sigma \in [\sigma_1 , \sigma_2 ]$ imply the other assumptions.


    We first prove the uniform convergence of $\widehat{\nu_{n, \sigma}} (y)$ for $|y|\leqslant a$ and $\sigma_1 \leqslant \sigma \leqslant \sigma_2 $ using an upper bound for $     \widehat{\nu_{n+1, \sigma}} (y) -     \widehat{\nu_{n, \sigma}} (y)$. By \eqref{eqn:nudef} we have
    \begin{equation}\label{nuhatform}
    \widehat{\nu_{n, \sigma}} (y)= \int_{[0,1]^n} e^{ i E_{n,\sigma}(\bth)\cdot y} \left| E'_{n,\sigma}(\bth)\right|^2 d\bth.
    \end{equation}
    From the definition of $L_n(\sigma, \chi_j; \bth )$ and $f_{n, \sigma, j}(\theta)$  in \eqref{def:L,theta} and \eqref{def:f}, respectively, we have
    \begin{equation*}
     E_{n+1, \sigma} (\bth_{n+1})  = \sum_{j=1}^J a_j L_{n+1}  (\sigma, \chi_j ; \bth_{n+1} ) 
      = \sum_{j=1}^J a_j L_n  (\sigma, \chi_j ; \bth)f_{n+1, \sigma, j} ( \theta_{n+1}) 
    \end{equation*}
   for $ \bth_{n+1} = ( \bth, \theta_{n+1})= ( \theta_1 , \dots,\theta_n,  \theta_{n+1} ) \in [0,1]^{n+1}$.  
       We shall only consider the case
    $$f_{n+1, \sigma, j }(\theta_{n+1}) =\left( 1- \frac{ w_{j,n+1} e^{2 \pi i \theta_{n+1}}}{(p_{n+1})^\sigma} + \frac{ e^{4 \pi i \theta_{n+1}}}{(p_{n+1})^{2\sigma}} \right)^{-1}$$
    in \eqref{def:f} for $ w_{j,n} = 2 \R [ \chi_j (\p_n)] $ since the others are easier and may be handled in the same way. Writing  $ L_n = L_n ( \sigma, \chi_j ; \bth )$ and $ L'_n = L'_n ( \sigma, \chi_j ; \bth )$, we see that
    \begin{equation}\label{eqn:wjn}\begin{split}
     & E_{n+1, \sigma} (\bth_{n+1}) \\
      & =   \sum_{j=1}^J a_j L_n  (\sigma, \chi_j ; \bth) \left( 1 + \frac{ w_{j,n+1} e^{2 \pi i \theta_{n+1}}}{(p_{n+1})^\sigma}  \right)  + O \left( \frac{  \sum_j |  L_n  | }{ ( p_{n+1} )^{2\sigma} }  \right) \\
     &=  E_{n, \sigma} (\bth) + \frac{  e^{2 \pi i \theta_{n+1} } } {(p_{n+1})^\sigma} \sum_{j=1}^J a_j  w_{j,n+1}   L_n  (\sigma, \chi_j ; \bth) + O \left( \frac{  \sum_j  |  L_n  | }{ ( p_{n+1} )^{2 \sigma} }  \right),
    \end{split} \end{equation}
    and similarly 
        \begin{equation*}\begin{split}
     & E'_{n+1, \sigma} (\bth_{n+1}) \\
      & =  E'_{n, \sigma} (\bth)  + \frac{  e^{2 \pi i \theta_{n+1} } } {(p_{n+1})^\sigma} \sum_{j=1}^J a_jw_{j,n+1}( L'_n (\sigma, \chi_j ; \bth)  - \log p_{n+1} L_n (\sigma, \chi_j ; \bth) ) \\
      & + O \left( \frac{ \sum_j ( |L_n| \log p_{n+1} + | L'_n| ) }{ ( p_{n+1})^{2 \sigma}}  \right).
            \end{split} \end{equation*}
 From the latter we find that
     \begin{equation*}\begin{split}
      |  E'_{n+1, \sigma} &(\bth_{n+1}) |^2  \\
        = &|  E'_{n, \sigma} (\bth) |^2  
         +2 \R \left[ \overline{ E'_{n,\sigma} (\bth)  }   \frac{  e^{2 \pi i \theta_{n+1} } } {(p_{n+1})^\sigma} \sum_{j=1}^J a_jw_{j,n+1}( L'_n (\sigma, \chi_j ; \bth)  - \log p_{n+1} L_n (\sigma, \chi_j ; \bth) )  \right]  \\
  & + R_{n, \sigma}(\bth),
   \end{split} \end{equation*}
 where
  \begin{equation*}   
R_{n, \sigma}(\bth) \ll \frac{ (1+\sum_j |L'_n | )( \sum_j |L_n| \log p_{n+1} + | L'_n| ) }{ ( p_{n+1})^{2 \sigma}} + \frac{  ( \sum_j |L_n| \log p_{n+1} + |L'_n| )^2 }{ ( p_{n+1})^{4\sigma}} .
      \end{equation*}
           Hence, we have
   \begin{equation}\label{eqn:bded 1}\begin{split}
    \widehat{\nu_{n+1, \sigma}} (y) = & \int_{[0,1]^n}\int_0^1 e^{iE_{n+1, \sigma}(\bth_{n+1})\cdot y}  
      |E'_{n, \sigma}(\bth)|^2 d\theta_{n+1} d\bth \\
      &+  \int_{[0,1]^n}\int_0^1 e^{iE_{n+1, \sigma}(\bth_{n+1})\cdot y}
     \frac{2}{(p_{n+1})^\sigma} \R \bigg(\overline{E'_{n, \sigma}(\bth)}e^{2 \pi i \theta_{n+1}}   \\ &   
    \qquad \qquad\sum_{j=1}^J a_j w_{j,n+1} ( L'_n (\sigma, \chi_j ; \bth)  - \log p_{n+1} L_n (\sigma, \chi_j ; \bth) ) \} \bigg) 
    d\theta_{n+1} d\bth \\
      &+  O\left( \int_{[0,1]^n}  R_{n, \sigma}(\bth)  d\bth \right)
    \end{split}\end{equation}
We apply Lemma \ref{fin:1} to the error term. Let $K=2$ and define $ c(m,1)$ and $c(m,2)$ for $ m \leqslant n $ by 
    \begin{equation*}
c(m,1)  =
\begin{cases}
 \chi_j ( \p _m )    & if\ \    p_m = \p_m ^2 , \cr
 1                  & if\ \    p_m = \p_m ,     \cr
 \chi_j ( \p _m )  & if\ \  p_m = \p_m \p'_m , \p_m \neq \p'_m \cr
\end{cases}
\end{equation*}
and 
    \begin{equation*}
c(m,2)  =
\begin{cases}
 0    & if\ \    p_m = \p_m ^2 , \cr
 -1                  & if\ \    p_m = \p_m ,     \cr
 \overline{ \chi_j ( \p _m ) }  & if\ \  p_m = \p_m \p'_m , \p_m \neq \p'_m .\cr
\end{cases}
\end{equation*}
Then we have 
    \begin{equation*}
    \mathcal{I} (\z, \w ) = \int _{ [0,1]^n}
      L_n ( \sigma + z , \chi_j ; \bth)\overline{  L_n (
    \sigma+ \overline{w}, \chi_j ; \bth )} d \bth
    \end{equation*}
    for $ \z = ( z,z) $ and $ \w = (w,w )$. By Lemma \ref{fin:1}, we have
    \begin{equation}\label{eqn:lem241}
    \mathcal{I} (0,0 ) = \int_{[0,1]^n} | L_n ( \sigma, \chi_j ; \bth ) |^2 d\bth \leqslant A 
    \end{equation}
    and 
    \begin{equation}\label{eqn:lem242}
    \frac{ \partial ^2 \mathcal{I}}{ \partial z \partial w } (0,0) = \int_{[0,1]^n} | L'_n ( \sigma, \chi_j ; \bth ) |^2 d\bth \leqslant A  
    \end{equation}
  for some $A>0$. Hence, by the Cauchy-Schwarz inequality we have
  \begin{equation*}  
\int_{[0,1]^n}  R_{n, \sigma}(\bth)  d\bth  
\ll \frac{ \log p_{n+1}  }{ ( p_{n+1})^{2 \sigma}}.
\end{equation*}  
We now wish to show that $\widehat{\nu_{n+1, \sigma}} (y) -\widehat{\nu_{n, \sigma}} (y)$ is small. 
By \eqref{eqn:bded 1} and our estimate for the error term, we have
    \begin{equation}\label{eqn:bded 4}\begin{split}
     & \widehat{\nu_{n+1, \sigma}} (y) -\widehat{\nu_{n, \sigma}} (y) \\ 
         = & \int_{[0,1]^n}\int_0^1 ( e^{iE_{n+1, \sigma}(\bth_{n+1})\cdot y}- e^{iE_{n, \sigma}(\bth )\cdot y} )  |E'_{n, \sigma}(\bth)|^2 d\theta_{n+1} d\bth \\
           &+    \frac{2}{(p_{n+1})^\sigma} \int_{[0,1]^n}\int_0^1 e^{iE_{n+1, \sigma}(\bth_{n+1})\cdot y}
   \R \bigg(\overline{E'_{n, \sigma}(\bth)}e^{2 \pi i \theta_{n+1}}   \\ &   
    \qquad \qquad\sum_{j=1}^J a_j w_{j,n+1} ( L'_n (\sigma, \chi_j ; \bth)  - \log p_{n+1} L_n (\sigma, \chi_j ; \bth) ) \} \bigg) 
    d\theta_{n+1} d\bth  \\
    &+  O\left( \frac{ \log p_{n+1}  }{  (p_{n+1 } )^{2 \sigma}} \right).
\end{split} \end{equation}
To treat the second term on the right-hand side above, we note that by
  \eqref{eqn:wjn}
      \begin{equation}\label{eqn:bded 2}\begin{split}
      \int_0^1  &e^{iE_{n+1, \sigma}(\bth, \theta_{n+1})\cdot y} e^{\pm 2 \pi i \theta_{n+1}} d\theta_{n+1} \\
     & =\int_0^1  e^{iE_{n, \sigma}(\bth)\cdot y} e^{\pm 2 \pi i \theta_{n+1}}
     \left( 1+ O\left( \frac{\sum_j|L_n|}{(p_{n+1})^\sigma}\right) \right) d\theta_{n+1} \ll \frac{\sum|L_n|}{(p_{n+1})^\sigma} .
    \end{split}\end{equation}
   To treat the first term   we use
    \begin{equation}\label{eqn:bded 3}\begin{split}
      \int_0^1 &e^{iE_{n+1, \sigma}(\bth_{n+1})\cdot y} -e^{iE_{n, \sigma}(\bth)\cdot y} d\theta_{n+1} \\ = & i e^{iE_{n, \sigma}(\bth)\cdot y}\int_0^1  y \cdot \left(e^{2 \pi i \theta_{n+1}} \sum_j a_j L_n (\sigma, \chi_j ; \bth)\frac{w_{j,n+1} }{(p_{n+1})^\sigma} \right) d\theta_{n+1} +O\left( \frac{ \sum |L_n |+|L_n |^2 }{(p_{n+1})^{2 \sigma}}\right) \\
      &\ll \frac{ \sum |L_n |+|L_n |^2 }{(p_{n+1})^{2 \sigma}}.
    \end{split}\end{equation}
    for $|y| \leqslant a$. 
    By \eqref{eqn:bded 4}--\eqref{eqn:bded 3}, we now see that
    \begin{equation*}\begin{split}
     & \widehat{\nu_{n+1, \sigma}} (y) -\widehat{\nu_{n, \sigma}} (y) \\ 
      &\ll \frac{1}{(p_{n+1})^{2\sigma}}  \int_{[0,1]^n} \left(\sum_j |L_n | + |L_n |^2 \right) |E'_{n,\sigma}(\bth) |^2  d\bth \\
      &+\frac{1}{(p_{n+1})^{2\sigma}}  \int_{[0,1]^n} \left( \sum_j |L_n| \right) |E'_{n,\sigma}(\bth)| \left( \sum_j |L'_n | + \log p_{n+1} |L_n | \right) d \bth + \frac{ \log p_{n+1} }{ (p_{n+1})^{2 \sigma}}.
     \end{split}\end{equation*}
Similarly to the equations \eqref{eqn:lem241} and \eqref{eqn:lem242}, we can apply Lemma \ref{fin:1} to see that
$$ \int_{[0,1]^n } |L_n ( \sigma, \chi_j ; \bth )|^{2\kappa } d\bth \leqslant A_\kappa $$ 
and 
$$ \int_{[0,1]^n } |L'_n ( \sigma, \chi_j ; \bth )|^{2\kappa } d\bth \leqslant A_\kappa $$ 
for any  positive integer $\kappa$.  Using these bounds and the Cauchy-Schwartz inequality, we obtain
    \begin{equation*}
    \widehat{\nu_{n+1, \sigma}} (y) - \widehat{\nu_{n, \sigma}} (y)    
   \ll \frac{\log p_{n+1}}{ (p_{n+1})^{2\sigma}}  
    \end{equation*}
    uniformly for $|y| \leqslant a$. It follows that for any $m>n$
    \begin{equation*}
    \widehat{\nu_{m, \sigma}} (y)-\widehat{\nu_{n, \sigma}} (y) \ll \sum_{p > p_n }^\infty \frac{\log p}{ p^{2\sigma}} 
    \ll  (p_n)^{1-2 \sigma}.
    \end{equation*}
    Therefore, $\widehat{\nu_{n, \sigma}} (y)$ converges uniformly for $|y| \leqslant a$ and $\sigma_1 \leqslant \sigma\leqslant \sigma_2$ as $n \to \infty$.

    Finally, we prove that there are $K, d, n' , \epsilon>0$ satisfying $|\widehat{\nu_{n,\sigma}}(y)| \leqslant K |y|^{-(2+\epsilon)}$ for $|y| \geqslant d $ and all $n \geqslant n'$.     By  \eqref{eqn:ep1} and \eqref{nuhatform}  we have
    \begin{equation}\label{eqn:3}\begin{split}
    \widehat{\nu_{n,\sigma}}(y) = & \int_{[0,1]^n} e^{iE_{n,\sigma}(\bth)\cdot y} |E'_{n,\sigma} (\bth)|^2 d\bth \\
    = & \sum_{j=1}^J |a_j|^2 \int_{[0,1]^n} e^{iE_{n,\sigma}(\bth)\cdot y} |L'_n (\sigma ,\chi_j ; \bth)|^2 d\bth \\
     & +\sum_{j \neq l} \bar{a_j}a_l \int_{[0,1]^n} e^{iE_{n,\sigma}(\bth)\cdot y} \overline{ L'_n (\sigma ,\chi_j ; \bth)}L'_n (\sigma ,\chi_l ; \bth) d\bth .
    \end{split}\end{equation}
   To bound this we  follow  the theory developed in Chapter 2 of 
    Borchsenius and Jessen~\cite{BJ}, which requires a number of definitions.

    For $\bth \in [0,1]^n$, we define  
    $$\L_{n, \sigma} (\bth) = ( L_n ( \sigma, \chi_1 ; \bth), \dots, L_n (\sigma, \chi_J ; \bth))$$
     and 
     $$\M_{n, \sigma} (\bth) = ( \log L_n ( \sigma, \chi_1 ; \bth) , \dots, \log L_n (\sigma, \chi_J ; \bth)).$$
    Let $\B \subseteq \mathbb{C}^J$ be a Borel set, let $ 1 \leqslant j,l \leqslant J$, $ j \neq l$, and let 
    $\sigma > \frac{1}{2}$. We define  distribution functions   
    \begin{equation*}\begin{split}
    \bla_{n, \sigma; j}(\B) = & \int_{\M_{n,\sigma}^{-1}(\B)} \left| \frac{L'_n}{L_n}(\sigma, \chi_j ; \bth) \right|^2 d\bth ,\\
    \bla_{n, \sigma; j,l ;\delta }(\B) = & \int_{\M_{n,\sigma}^{-1}(\B)} \left| \frac{L'_n}{L_n}(\sigma, \chi_j ; \bth) +\delta \frac{L'_n}{L_n}(\sigma, \chi_l ; \bth) \right|^2 d\bth
    \end{split}\end{equation*}
    for $\delta = \pm 1 , \pm i $,  and their Fourier transforms as
    \begin{equation*}\begin{split}
    \widehat{\bla_{n, \sigma; j } }(\y) = & \int_{\mathbb{C}^J} e^{i \x \cdot \y} d \bla_{n, \sigma; j}(\x), \\
    \widehat{\bla_{n, \sigma; j,l ;\delta } }(\y) = & \int_{\mathbb{C}^J} e^{i \x \cdot \y} d \bla_{n, \sigma; j,l; \delta}(\x),
    \end{split}\end{equation*}
    where $ \x = ( x_1 +ix_1', \dots , x_J+ix_J' ) $, $\y = (y_1 +iy_1' , \dots , y_J+iy_J')$ and $ \x \cdot \y = \sum_{j=1}^J ( x_j y_j + x_j' y_j' ) $. Then we have
    \begin{equation*}\begin{split}
    \widehat{\bla_{n, \sigma; j } }(\y) = & \int_{[0,1]^n} e^{i\M_{n,\sigma}(\bth)\cdot \y} \left| \frac{L'_n}{L_n}(\sigma, \chi_j ; \bth) \right|^2 d\bth , \\
    \widehat{\bla_{n, \sigma; j,l ;\delta} }(\y) = & \int_{[0,1]^n} e^{i\M_{n,\sigma}(\bth)\cdot \y} \left| \frac{L'_n}{L_n}(\sigma, \chi_j ; \bth)+\delta \frac{L'_n}{L_n}(\sigma, \chi_l ; \bth) \right|^2 d\bth.
    \end{split}\end{equation*}
    Since $\log L_n (\sigma, \chi_h ; \bth) = \sum_{k=1}^n \log f_{k, \sigma, h}(\theta_k)$ and $ L'_n/L_n(\sigma, \chi_h ; \bth) = \sum_{k=1}^n   f'_{k, \sigma, h}/f_{k, \sigma, h} (\theta_k)$, we have
    \begin{equation}\label{fut:1}
    \widehat{\bla_{n, \sigma; j } }(\y) = \sum_{m=1}^n K_{2, m}(\y) \prod_{k \neq m} K_{0,k} (\y) + \sum_{m \neq m'} \overline{ K_{1, m}(-\y)} K_{1, m'}(\y) \prod_{k \neq m, m'} K_{0,k} (\y),
    \end{equation}
    where
    \begin{equation*}\begin{split}
    K_{0,k} (\y) & = \int_0 ^1 e^{i \sum_{h=1}^J \log f_{k, \sigma , h} (\theta) \cdot (y_h +i y_h') } d\theta, \\
    K_{1,k} (\y) & = \int_0 ^1 e^{i \sum_{h=1}^J \log f_{k, \sigma , h} (\theta) \cdot (y_h +i y_h') } \frac{f'_{k,\sigma,j}}{f_{k,\sigma,j}}(\theta) d\theta ,\\
    K_{2,k} (\y) & = \int_0 ^1 e^{i \sum_{h=1}^J \log f_{k, \sigma, h} (\theta) \cdot (y_h +i y_h') } \left| \frac{f'_{k,\sigma,j}}{f_{k,\sigma,j}}(\theta)\right|^2 d\theta.
    \end{split}\end{equation*}
    As an easy consequence of the definitions of the $ K_{i,k} (\y), i=0, 1, 2$, we see that
       \begin{align}
    |K_{0,k}(\y)| & \leqslant 1 ,\label{max:1}  \\
    |K_{2,k}(\y)| & \leqslant C \frac{\log ^2 p_k }{p_k^{2 \sigma}},\label{max:4} 
    \end{align}
    and
    \begin{equation}\label{fut:2}
    K_{1,k}(\y) =  \int_0 ^1 \left(  1+ O\left( \frac{ \| \y \|}{p_k ^\sigma}\right)\right) 
    \frac{f'_{k,\sigma,j}}{f_{k,\sigma,j}}(\theta) d\theta  
    \ll \frac{ \| \y \| \log p_k }{p_k ^{2\sigma}},
    \end{equation}
    where $C$ is a positive absolute constant.
    By Lemma \ref{curve:1} we also have 
    \begin{equation}\label{max:2}
    |K_{0,k} (\y)| \leqslant \frac{Cp_k ^{\sigma/2}}{\sqrt{ \| \y \|}}
    \end{equation}
   provided that $p_k \in P_2$ is sufficiently large and $\y \in \Bbb{C}^J$ satisfies
   \begin{equation}\label{sum condition}
   \bigg| \sum_{h=1}^J (\R \chi_h (\p_k )) ( y_h  + i y_h' ) \bigg| \geqslant  \| \y \|  / 7;
  \end{equation}
  here $\p_k$ is either of the two primes ideals lying above $p_k$. Note that 
  if $p_k=\p_k \p_k'$, then $\R \chi_h (\p_k ) =\R \chi_h (\p_k' )$, so the value of the sum is the same for either.
  We wish to show that for each $\y$ this condition is met by the $\p_k$'s in at least one
  of the $h(D)$ ideal classes.  To this end,  we consider the sum 
  $ \sum_{ \p } \big| \sum_{h=1}^J (\R \chi_h (\p)) b_h \big|^2 $, 
  where the sum is over one representative prime ideal from each of the $h(D)$ ideal classes and 
  $ b_h \in \Bbb{C}$.  We note that $ \chi_h \neq \chi_{h'} $ and 
  $ \chi_h \neq \bar{\chi}_{h'} $ for $ h \neq h' $. By letting $b_\chi = b_h $ 
  for a real character $\chi = \chi_h $  and $b_\chi =  b_h / 2$ for a non real 
  character $\chi= \chi_h $ or $ \chi = \bar{\chi}_h $, we see that
    \begin{equation*}
      \sum_{ \p } \bigg| \sum_{h=1}^J (\R \chi_h (\p ) ) b_h \bigg|^2 =  
      \sum_{ \p } \bigg| \sum_\chi  \chi (\p ) b_\chi \bigg|^2 = 
       h(D) \sum_\chi |b_\chi|^2 \geqslant \frac{h(D)}{2} \sum_{h=1}^J |b_h|^2.
    \end{equation*}
    Thus we have
    \begin{equation*} 
    \max_{\mathfrak{p}} \bigg|\sum_{h=1}^J (\R \chi_h ( \p )) b_h \bigg|^2 
    \geqslant  \frac{1}{2} \sum_{h=1}^J |b_h|^2.
    \end{equation*}
   From this we see that  \eqref{sum condition} holds for at least one prime ideal $\p$.
    
   Returning to \eqref{fut:1},  for a given $\y$  consider the product 
   \begin{equation*}
   \prod_{\substack{k=1\\ k\neq m}}^n  K_{0, k}(\y).
     \end{equation*}
 To each $k$ there corresponds a prime $p_k$, and by Lemma \ref{tauber:1}  
 we see that as $n\to\infty$, the number of these that are in $P_2$ is $\sim n$.
 Again by  Lemma \ref{tauber:1} it is clear that  a positive proportion of  these $p_k$'s
will split into prime ideals for which \eqref{sum condition} holds.
Thus, by (\ref{max:1}) and (\ref{max:2}) we have  
$$
 \prod_{\substack{k=1\\ k\neq m}}^n  |K_{0, k}(\y)|  
\leq A_r\bigg(\frac{ 1}{\sqrt{ \| \y \|}}\bigg)^r
$$
for any positive integer $r$, provided that $n$ is large enough.   Obviously the same estimate holds for 
the other product in \eqref{fut:1} as well.     
   Thus, given $q>0$, there is a number 
    $n_q$ such that for $n \geqslant n_q$,
    \begin{equation*}
    \widehat{\bla_{n,\sigma ; j}}(\y) = O( \|\y\|^{-(2k +q+1)}).
    \end{equation*}
    Similarly, we have
    \begin{equation*}
    \widehat{\bla_{n,\sigma ; j, l; \delta}}(\y) = O( \|\y\|^{-(2k+q+1)}).
    \end{equation*}
    
  By the discussion in the beginning of \S 2, we now see that  the distributions $ \bla_{n,\sigma ; j}$ and $\bla_{n,\sigma ; j, l; \delta}$ are absolutely continuous, and their densities $\F_{n, \sigma ; j}(\x)$, $\F_{n, \sigma; j, l;\delta}(\x)$ are continuous and possesses continuous partial derivatives of order $\leqslant q$ for $n \geqslant n_q$. 
  
By the definitions of $K_{0,k}(\y)$ and  $f_{k, \sigma , h}(\theta)$ we find that
  $$K_{0,k}(\y) = \int_0 ^1 1+ i \sum_{h=1}^J \log f_{k, \sigma , h}(\theta) 
  \cdot (y_h+iy_h') d\theta + O\left( \frac{ \|\y\|^2}{p_i ^{2 \sigma}}\right) = 1+ O\left( \frac{ \|\y\|^2}{p_i ^{2 \sigma}}\right).$$ 
Using this and   (\ref{fut:1})--(\ref{fut:2}), we see that 
 $$\widehat{ \bla_{n+1, \sigma; j}}(\y) - \widehat{ \bla_{n, \sigma; j}}(\y)\ll \frac{ \log^2 p_{n+1}}{p_{n+1} ^{2 \sigma}} $$
  uniformly for every sphere $\|\y\| \leqslant a$.  By Cauchy's convergence criterion, it follows that 
  $\lim_{n \to \infty} \widehat{ \bla_{n, \sigma; j}}(\y)$ exists, and by our discussion in \S 2 the limit  
  function is the Fourier transform of a distribution $\bla_{\sigma; j}$,  and $ \bla_{n, \sigma ; j} \to \bla_{ \sigma;j}$ as $n\to \infty$. Furthermore, the density $\F_{\sigma ; j}(\x)$ of $\bla_{\sigma; j}$ is the limit of the density functions $ \F_{n, \sigma ; j }(\x)$ for $\sigma > 1/2$. Analogous results  hold for $\bla_{\sigma; j, l ; \delta}$. By Theorem 6 of \cite{BJ}, for any $c>0$, the partial derivatives of the densities $\F_{\sigma; j}(\x)$, $ \F_{n, \sigma; j}(\x)$, $ \F_{\sigma; j,l ; \delta}(\x)$ and $ \F_{n, \sigma; j,l; \delta}(\x)$ of order $\leqslant q$ have a majorant $K_q e^{- c \|\x \|^2}$ for $n \geqslant n_q$.
    
    Next we define $\bla_{n,\sigma; j, l}$ by
    $$\bla_{n,\sigma; j, l}(\B) = \frac{1}{4}\sum_{\delta= \pm 1, \pm i } \bar{\delta} \ \bla_{n, \sigma; j, l ; \delta}(\B)$$
    for any Borel set $\B \subseteq \mathbb{C}^J$.
   Then by the identity  $\bar{a}b = \frac{1}{4} ( |a+b|^2 - |a-b|^2 -i |a+ib|^2 +i |a-ib|^2 )$ and the definition of 
   $\bla_{n, \sigma; j, l ; \delta}(\B)$, we have
    \begin{equation*}
    \bla_{n,\sigma; j, l}(\B)= \int_{\M^{-1}_{n, \sigma}(\B)} \overline{ \frac{L'_n}{L_n}(\sigma, \chi_j ; \theta) }\frac{L'_n}{L_n}(\sigma, \chi_l ; \theta) d\bth.
    \end{equation*}
    Clearly the density function of $\bla_{n,\sigma; j,l}$ is $\F_{n, \sigma; j, l}(\x)= \frac{1}{4}\sum_{\delta= \pm 1, \pm i } 
    \bar{\delta} \F_{n, \sigma; j, l ; \delta}(\x)$, and its partial derivatives of order $\leqslant q$ are majorized by 
    $K_q e^{-c \|\x\|^2}$ for $n \geqslant n_q$.

We are now ready to estimate $\widehat{\nu_{n,\sigma}}(y)$  in  \eqref{eqn:3}. There are two sums on the right-hand
side of  \eqref{eqn:3}. We will just estimate the terms in the second sum as those in the first are similar. 
    For $\B \subseteq \mathbb{C}^J$, we define $\B_{\log} = \{ \x \in \mathbb{C}^J : e^\x \in \B \}$, where $\x = ( x_1 + ix_1' , \dots, x_J+ix_J' )$ and $ e^\x = ( e^{x_1 + i x_1 ' } , \dots e^{ x_J + i x_J' } )$. Then we have $\M_{n,\sigma}^{-1}( \B_{\log} ) = \L_{n,\sigma}^{-1}( \B )$. We also have $ E_{n, \sigma}(\bth )\cdot y = \L_{n, \sigma}(\bth) \cdot \z$ for $ \z = (\bar{a_1} y, \dots , \bar{a_J}y) \in \mathbb{C}^J$. The typical term  in the second sum   in \eqref{eqn:3} is 
    \begin{equation*}\begin{split}
    I=
    & \int_{[0,1]^n} e^{i E_{n, \sigma}(\bth)\cdot y} \overline{ L'_n (\sigma, \chi_j ;\bth)} L'_n ( \sigma, \chi_l; \bth) d\bth \\
    = & \int_{[0,1]^n} e^{i \L_{n, \sigma}(\bth)\cdot \z} \overline{ L_n (\sigma, \chi_j; \bth)} L_n (\sigma, \chi_l ;\bth) \overline{ \frac{L'_n}{L_n} (\sigma, \chi_j ;\bth)} \frac{L'_n}{L_n} ( \sigma, \chi_l; \theta) d\bth \\
    = & \int_{\mathbb{C}^J} e^{i e^\x \cdot \z + x_j- ix_j' +x_l + ix_l' } d\bla_{n,\sigma; j,l}(\x)
    =  \int_{\mathbb{C}^J} e^{i e^\x \cdot \z + x_j- ix_j' +x_l + ix_l' } \F_{n,\sigma; j,l}(\x) d\x,
    \end{split}\end{equation*}
    where $ d\x =  \prod_{h=1}^J dx_h dx_h'$.  
   Let  $ \m = ( m_1 , \dots, m_J )$ and $ R_m = \{ x+ix' : x \in \mathbb{R}, x' \in [ 2 \pi m , 2 \pi (m+1) ] \} $.
  Then we can write $ \mathbb{C}^J = \bigcup_{ \m \in \mathbb{Z}^J} \prod_{h=1}^J R_{m_h} $.  We then have
   \begin{equation*}\begin{split}
   I=  & \sum_{ \m \in \mathbb{Z}^J}  \int_{   R_{m_1} \times \cdots\times R_{m_J}}  e^{i e^\x \cdot \z + x_j- ix_j' +x_l + ix_l' } \F_{n,\sigma; j,l}(\x) d\x \\
   = & \sum_{ \m \in \mathbb{Z}^J}  \int_{ R_{0} \times \cdots\times  R_0}  e^{i e^\x \cdot \z + x_j- ix_j' +x_l + ix_l' } \F_{n,\sigma; j,l}(\x - 2 \pi i \m ) d\x .
   \end{split} \end{equation*}
   We substitute $ e^{x_h} = r_h$ and $x_h'= \theta_h $ for each $h = 1, \dots, J$. 
   Letting $\r = ( \log r_1  +i \theta_1  , \dots , \log r_J +  i \theta_J  ) $ and 
   $ R_+ =\{ (r,\theta) : r>0, \theta  \in [ 0, 2 \pi ] \}    $  we find that
   \begin{equation*} 
   I=  \sum_{ \m \in \mathbb{Z}^J}  \int_{ R_+^J }  e^{i  \r \cdot \z } r_j r_l e^{ -i\theta_j +i\theta_l } 
   \left(\prod_{h=1}^J r_h ^{-1}\right) \F_{n,\sigma; j,l}( \r - 2 \pi i \m ) d\r,
     \end{equation*}
  where $d\r = \prod_{h=1}^J dr_h d\theta_h $.  We carry out the summation and integrations sequentially over the triples
  $(m_1, r_1, \theta_1), (m_2, r_2, \theta_2),  \ldots (m_J, r_J, \theta_J)$.  
  Treating $(m_1, r_1, \theta_1)$ first, we see that we must estimate
      \begin{equation*}\begin{split}
     & \sum_{m_1 \in \mathbb{Z}} \int_0 ^{2\pi} \int_0 ^\infty  e^{i  r_1 (e^{i\theta_1}\cdot \overline{a}_1y  )}  
     \F ( \log r_1 + i  (\theta_1- 2 m_1 \pi ), \ldots)
    \frac{dr_1}{r_1^k} d\theta_1 \\
    = & \int_{-\infty} ^{\infty} \int_0 ^\infty  e^{ir_1| {a}_1 y|  \sin (\theta_1 -\alpha )}   
    \F ( \log r_1 + i  (\theta_1- 2 m_1 \pi ), \ldots)
    \frac{dr_1}{r_1^k} d\theta_1,    \end{split}\end{equation*}
    where $k=0$ or $1$ and $\alpha $ depends only on the argument of $\overline{a}_1y$. 
    Let $I_\epsilon = \bigcup_{n \in \mathbb{Z}} [ \alpha +
    n \pi - \epsilon, \alpha + n \pi + \epsilon]$ for $ 0<\epsilon < \pi$.
    Then repeated integration by parts with respect to $r_1$ over $\mathbb R\setminus I_\epsilon$ 
    using the majorant $K_q e^{-\lambda(\log^2 r_1 + \theta_1^2)}$ for the 
    partial derivatives of $\F$ of order $\leqslant q$ gives    
   \begin{equation*}
   \begin{split}
   \int_{\mathbb{R} \setminus I_\epsilon}  \int_0 ^\infty  &e^{ir_1|a_1 y|  \sin (\theta_1 -\alpha)} 
    \F ( \log r_1 + i  (\theta_1- 2 m_1 \pi ), \ldots)
    \frac{dr_1}{r_1^k}  d\theta_1  \\
       &\ll \int_{\mathbb{R} \setminus I_\epsilon} |y \sin (\theta_1 - \alpha) 
       |^{-q}e^{-\lambda \theta_1^2} d\theta_1\\
      & \ll |\epsilon y |^{-q}  . 
       \end{split}
      \end{equation*}
    For $\theta \in I_\epsilon$, we change the order of integration and integrate by parts over 
    $\theta_1$ once and see that
    $$ \int_0 ^\infty \sum_{m \in \mathbb{Z}} 
    \int_{\alpha + m\pi - \epsilon}^{\alpha + m \pi + \epsilon} e^{ir_1|a_1 y|  \sin (\theta_1 -\alpha)} 
    \F ( \log r_1 + i  (\theta_1- 2 m_1 \pi ), \ldots)d\theta_1  \frac{dr_1}{r_1^k}  \ll |y|^{-1} .$$
    Thus, we have
    \begin{equation*}
   \sum_{m_1 \in \mathbb{Z}} \int_0 ^{2\pi} \int_0 ^\infty  e^{i  r_1 (e^{i\theta_1}\cdot \overline{a}_1y  )}  
     \F ( \log r_1 + i  (\theta_1- 2 m_1 \pi ), \ldots)
    \frac{dr_1}{r_1^k} d\theta_1 \ll |y|^{-1}.
    \end{equation*}
 We  continue  this process for the remaining triples $(m_i, r_i, \theta_i), \ i=2, \ldots, J,$ 
 and at each stage we get another factor of $|y|^{-1}$.
 In this way we   find that 
     \begin{equation*}
 I=   \int_{[0,1]^n} e^{i E_{n, \sigma}(\bth)\cdot y} \overline{ L'_n (\sigma, \chi_j ;\bth)} 
 L'_n ( \sigma, \chi_l; \bth) d\bth  \ll |y|^{-J}
    \end{equation*}
    as $|y| \rightarrow \infty$ for $ n \geqslant n_1 $.  Using  this in \eqref{eqn:3}, 
    we see that $\widehat{\nu_{n,\sigma}}(y)  \ll |y|^{-J} \ll |y|^{-3}$ as $|y| \rightarrow \infty$ for $n \geqslant n_1$. This completes the proof of Theorem \ref{thm:1}.

    We now turn to the proof of Theorem \ref{co:1}.
    \begin{proof}[Proof of Theorem \ref{co:1}]
    Recall that by equations (\ref{eqn:4}) and (\ref{eqn:5}) 
    $$ \lim_{T \rightarrow \infty} \frac{1}{T} N_E ( \sigma_1 ,     \sigma_2 ; 0,T) =  \int_{\sigma_1}^{\sigma_2}
    G_\sigma  d\sigma  $$
    for $1/2< \sigma_1 < \sigma_2 $ and  some function $G_\sigma $ continuous  for $\sigma>1/2$. 
    We  denote by $N_E^{line} (\sigma; T)$  the number of zeros $\rho$ of $E(s,Q)$ with  $\R \rho = \sigma $ and $ 0 < \I \rho <T$. Then we have
    $$  \int_{\sigma_0-\epsilon }^{\sigma_0+\epsilon} G_\sigma  d\sigma = \lim_{T \to \infty} \frac{1}{T}  N_E ( \sigma_0 - \epsilon ,  \sigma_0 + \epsilon ; 0,T) \geqslant  \limsup_{T \to \infty} \frac{1}{T}N_E ^{line} ( \sigma_0 ; T) \geqslant 0 $$
    for any $ \epsilon >0$. Taking $\epsilon \to 0+$, we  then see that 
    $$ \limsup_{T \rightarrow \infty}\frac1T N_E ^{line}( \sigma_0 ; T)=0.$$
 Thus, $N_E^{line} (\sigma_0 ; T) = o(T)$.
    \end{proof}

\section{Zeros of Epstein zeta functions in $\R s > 1 $ }

As we mentioned in \S 1, we expect that the constant $c$ in
Theorem~\ref{thm:1} is positive when
$1< \sigma_1 <\sigma_2 < \sigma(Q)$ for all positive definite binary
quadratic forms $Q$ with integer coefficients, and not just  for $Q_1$ and $Q_2$.
We are not able to prove  this, but Corollaries \ref{cor:1} and \ref{cor:2} below 
provide partial results in this direction. These results will be consequences of Theorem~\ref{thm:1} 
and the two theorems below from the theory of almost periodic functions.
 The first theorem says that if an almost periodic function has a zero $\rho$, 
 then there are infinitely many zeros in 
 any narrow strip containing the zero. See the comments by Montgomery to 
 Davenport's collected works \cite{D}, Vol IV, p. 1780.

    \begin{thrm}\label{offzero:1}
        Let $ f(s) = \sum_{n=1}^\infty a_n n^{-s} $ for $\R s
        >1 $. Suppose that $f(\rho) = 0 $
        where $\rho = \beta +i \gamma $ and $\beta >1$. Let $
        0< \delta < \beta-1$ be fixed. Then  there are at least $c_\delta T$ zeros of $f(s)$
        in the region $ | \R s - \beta | < \delta $,\ $ 0< \I s <T$, where
        $c_\delta >0$ is a constant independent of $T$.
    \end{thrm}
    \begin{proof}
    Let $r$ be so small  that $0<r<\delta$ and that $\rho$ is the only zero 
    of $f$ in the disk $|s-\rho|\leq  r$.
    Also let $ \epsilon = \inf_{\theta} | f( \rho + re^{i\theta})|$.  
        By an extension of Dirichlet's theorem on Diophantine approximation (e.g., see $\S 8.2 $ of \cite{Ti}), we can find $m$ values of $t_0$ in the interval $[1, mT_\epsilon ]$     satisfying $| f(s+it_0 )- f(s)| < \epsilon$ for $ |\R s - \beta | < \delta$, with any two solutions differing by at least $1$. Thus, we have $|f(s+it_0 )-f(s) | < |f(s)| $ for $|s-\rho
    | = r$. By Rouch\'{e}'s theorem, $f(s)$ has at least one zero in $
    | s - \rho - it_0 |< r $.  Hence, we have at least
    $ m $ zeros in $ | \R s - \beta | < \delta $ and $ 1< \I s
    < m T_\epsilon$.
    \end{proof}

Assume the same hypothesis as in Theorem \ref{vo:1} and consider the constant $c= c(\sigma_1 , \sigma_2 , Q )$ in Theorem \ref{thm:1}. 
By Theorem \ref{thm:1}, Theorem~\ref{offzero:1}, and Davenport and Heilbronn's theorem 
(see below Theorem~\ref{vo:1}), we deduce that $\sigma(Q) > 1 $ and $ c( 1, \sigma, Q) >0$ for any $ \sigma >1 $. The definition of $\sigma(Q)$ and Theorem \ref{offzero:1} then implies  that 
$ c( \sigma, \sigma(Q) , Q) >0$ for any $ 1/2 < \sigma < \sigma(Q)$. 

Next let $ 1 < \sigma_1 < \sigma_2 < \sigma(Q)$ and suppose $c(\sigma_1, \sigma_2,Q) = 0$. If there exists a zero $\rho$ in the vertical strip $ \sigma_1 < \R s < \sigma_2 $, then Theorem \ref{offzero:1} implies 
that $ c(\sigma_1, \sigma_2 ,Q)>0$, and this contradicts   our assumption. Hence $E(s,Q)$ has no zeros in the vertical strip $ \sigma_1 < \R s < \sigma_2$. 

We summarize these results in the following corollary.
    \begin{co}\label{cor:1}
    Assume the same hypothesis as in Theorem \ref{vo:1} and let $c=c(\sigma_1, \sigma_2, Q)$ be
    the constant in Theorem
    \ref{thm:1}.  Then for any $ \sigma> 1$ we have $c(1, \sigma , Q)> 0 $,  and
    for any $1/2 < \sigma < \sigma(Q)$ we have 
    $ c(\sigma, \sigma(Q), Q) >0$.
   Moreover, if $c(\sigma_1 , \sigma_2, Q) =0$ for some $ 1 < \sigma_1 <
    \sigma_2 < \sigma(Q)$, then $ E(s, Q) \neq 0 $ for $\sigma_1 < \R
    s< \sigma_2$.
    \end{co}

The theory of mean motions explains the zeros of almost periodic functions by using the Jensen function. We restate Theorem 31 of \cite{JT} as follows.

    \begin{thrm}\label{thm:mm}
       Let  $ f(s) = \sum_{n=n_0}^\infty  a_n  n^{-s} $ with $a_{n_0} \neq 0$ and with the  abscissa 
       of uniform convergence $\alpha$. Then the Jensen function $ \varphi(\sigma)$ possesses on every half-line $\sigma > \alpha_1   > \alpha $ only a finite number of linearity intervals and a finite number of points of non-differentiability. The values of $\varphi'(\sigma)$ in the linearity intervals belong to the set of numbers $ - \log n $, $n \geqslant n_0$. For $ \sigma$ sufficiently large, we have
        $$ \varphi(\sigma) = - ( \log n_0 ) \sigma + \log | a_{n_0}|.$$
        For any $ \sigma > \alpha$ the mean motions $ c^- ( \sigma) $ and $c^+ (\sigma)$ both exist and are determined by $ c^- ( \sigma ) = \varphi '( \sigma  - )$ and $ c^+ ( \sigma ) = \varphi' ( \sigma+ )$.
       In any strip $ ( \sigma_1 , \sigma_2 )$, where $ \alpha< \sigma_1 < \sigma_2 < \infty$, the relative frequency $H( \sigma_1 , \sigma_2 ) $ of zeros exists and is determined by
        $$ H( \sigma_1 , \sigma_2 ) = \frac{1}{2\pi } ( \varphi'( \sigma_2 - ) - \varphi' ( \sigma _1 +)).$$
    \end{thrm}

As a consequence of Theorem \ref{thm:1} and \ref{thm:mm}  we have the following corollary.
\begin{co}\label{cor:2}
Assume the same hypothesis as in Theorem \ref{vo:1}. There are only finitely many zero free vertical strips for $E(s,Q)$.
\end{co}

\begin{proof}
By the functional equation it suffices to show that there are only finitely many zero free vertical strips 
in $\sigma>1/2$.
  First, by Theorem \ref{vo:1}   any  vertical strip $\sigma_1\leq \R s \leq \sigma_2$ with 
   $ 1/2 <\sigma_1<\sigma_2\leq 1$ has at least $cT$ zeros with $c>0$. Next consider strips in $ \R s > 1$. Theorem \ref{thm:mm} says that the number of zeros in the region 
   $ \sigma_1 < \R s <\sigma_2 $, $ 0<\I s < T$ is 
  $$ \frac{T}{2 \pi} ( \varphi'( \sigma_2 -) - \varphi'(\sigma_1 +)) + o(T), $$
 and that $\varphi( \sigma)$ has only a finite number of linearity intervals. On these linearity intervals  $\varphi'(\sigma)$ is constant and $\varphi'(\sigma_2 - ) - \varphi'(\sigma_1 + ) = 0 $. Thus, by Corollary \ref{cor:1}, $E(s,Q)$ has no zeros in the vertical strip corresponding to these linearity intervals. If $[\sigma_1, \sigma_2]$ is not in any of these linearity intervals, then $\varphi'(\sigma_2 - ) - \varphi'(\sigma_1 + ) >0$ since $ \varphi ( \sigma)$ is convex. This completes the proof.
\end{proof}

\begin{ack}
This paper is a part of my Ph.D. thesis at Yonsei University. I would like to express my gratitude to my advisor Professor Haseo Ki for his unfailing support. I really appreciate Professor Steve Gonek for helping me to revise this paper. I thank Professor Enrico Bombieri for his kindness and many invaluable suggestions and corrections to this paper, Professor Roger Heath-Brown for showing me his interest and suggestions for this paper, and the reviewer for his or her careful reading and many comments.
\end{ack}

\end{document}